\documentclass[11pt]{article}

\usepackage{amsfonts}
\usepackage{amscd}
\usepackage{amssymb}
\usepackage{amsthm}
\usepackage{amsmath, xspace}
\usepackage{blkarray}
\usepackage{bm}
\usepackage{cancel}
\usepackage{color}
\usepackage{graphics}
\usepackage{graphicx}
\usepackage{enumitem}
\usepackage{fancyhdr}
\usepackage{mathdots}
\usepackage{mathrsfs}
\usepackage{multicol}
\usepackage{stmaryrd}
\usepackage{ytableau}
\usepackage[all]{xy}

\usepackage[plainpages,backref]{hyperref}

\usepackage{lscape}

\theoremstyle{plain}
\newtheorem{thm}{Theorem}[section]

\newtheorem{prop}[thm]{Proposition}
\newtheorem{cor}[thm]{Corollary}
\theoremstyle{definition}

\theoremstyle{example}

\theoremstyle{remark}

\numberwithin{equation}{section}

\providecommand{\keywords}[1]{\textbf{\textit{Key words---}} #1}

\def\cC{\mathcal{C}}

\def\cP{\mathcal{P}}

\def\cW{\mathcal{W}}
\def\cX{\mathcal{X}}

\def\CC{\mathbb{C}}

\def\FF{\mathbb{F}}

\def\ZZ{\mathbb{Z}}

\def\Card{\mathrm{Card}}

\def\dim{\mathrm{dim}}
\def\End{\mathrm{End}}

\def\Ind{\mathrm{Ind}}

\def\tr{\mathrm{tr}}

\usepackage{etex} 
\usepackage{pictexwd}
\usepackage{tikz}
	\usepgflibrary[patterns] 
	\usetikzlibrary{patterns} 
	\usepgflibrary{shapes.geometric}
        \usetikzlibrary{arrows, calc, positioning}

\makeatletter
\renewcommand{\@makefnmark}{\mbox{\textsuperscript{}}}
\makeatother

\title{Lusztig varieties and Macdonald polynomials}
\author{
Arun Ram\quad\ \ email:\ aram@unimelb.edu.au \\
\\
}
\date{}

\usetikzlibrary{arrows.meta}

\begin{document}

\maketitle

\vspace{-2em}
\begin{center}
{\sl Dedicated to Peter Littelmann}
\end{center}

\begin{abstract}
\noindent
This paper uses Lusztig varieties to give central elements of the Iwahori-Hecke algebra
corresponding to unipotent conjugacy classes in the finite Chevalley group $GL_n(\FF_q)$.
We explain how these central elements are related to Macdonald polynomials and how this provides
a framework for generalizing integral form and modified Macdonald polynomials to Lie types other than $GL_n$.
The key steps are to recognize (a)  that counting points in Lusztig varieties is equivalent to computing traces on the Hecke algebras,
(b) that traces on the Hecke algebra determine elements of the center of the Hecke algebra, 
(c) that the Geck-Rouquier basis elements of the center of the Hecke algebra produce an `expansion matrix', 
(d) that the parabolic subalgebras of the Hecke algebra produce a `contraction matrix' and 
(e) that the combination `expansion-contraction' is the plethystic transformation that relates integral form Macdonald
polynomials and modified Macdonald polynomials.
\end{abstract}



\section{Introduction}

\subsection{Motivation for this paper}

Since their introduction by Garsia and Haiman \cite{GH93}, the modified Macdonald polynomials have blossomed into
a huge, exciting subfield of algebraic combinatorics with an extensive literature.
A wonderful survey of the history of the discovery of these amazing objects is in \cite{GR05}.

One of the challenges of Macdonald polynomial theory for this author has been that he did not know
(and did not have any good sense for) how to define modified Macdonald polynomials outside of type A.
This paper provides an answer.

In fact there are three things that need to be generalized to Lie types other than type $GL_n$,
\begin{enumerate}[topsep=0.2em,itemsep=-0.2em]
\item[(a)] the integral form Macdonald polynomials $J_\mu(x;q,t)$,
\item[(b)] the modified Macdonald polynomials $\tilde{H}_\mu[X;q,t]$, and
\item[(c)] the plethystic transformation that relates the integral form and the modified form.
\end{enumerate}
The route to generalization comes from three connections:
\begin{enumerate}[topsep=0.2em,itemsep=-0.2em]
\item[(a)] the monomial expansion of $J_\mu(x;q,t)$ counts points of affine Lusztig varieties,
\item[(b)] the monomial expansion of $\tilde{H}_\mu[X;q,t]$ counts points of parabolic affine Springer fibers, and
\item[(c)] the Lusztig varieties and parabolic Springer fibers are related by expansion-contractoin.
\end{enumerate}
\emph{Contraction} is the information of which permutations are in a parabolic (Young) subgroup -- how to `contract' the
symmetric group to a Young subgroup.
\emph{Expansion} is the information of which permutations are in which conjugacy classes --
how a conjugacy `expands' as a sum of permutations; 
except that we need this expansion process on the level of the Iwahori-Hecke algebra.
The plethystic transformation is the process of expansion followed by contraction.

\subsection{Plan of the paper}

Section 2 reviews the Iwahori-Hecke algebra and its action on the flag representation $\mathbf{1}_B^G$.  
We explain how counting the points of Lusztig varieties $Y^{-1}_{BwB}(g)$ and parabolic Springer fibers
$Y^{-1}_{P_\pi}(g)$ is captured by traces and central elements in the Hecke algebra 
(for an expanded point of view on this counting see \cite{Lu21} and references there).  Using these tools
we quickly derive a formula relating the number of points of the parabolic Springer fiber to the number
of points of a union of Lusztig varieties.  This is the geometric source of the expansion-contraction that 
plays the essential role in the proof of Theorem \ref{atob} in Section 4.

The Hall-Littlewood polynomials are a special case of Macdonald polynomials (obtained by specializing 
the variable $q$ to $0$).  In Section 3 we review the connection between integral form Hall-Littlewood polynomials
and the number of points of nilpotent Lusztig varieties which was established in \cite[Theorem 4.11]{HR99}.
Then we review the connection between modified Hall-Littlewood polynomials and the number of points
in parabolic Springer fibers. 

In Section 4 we prove that the integral form Macdonald polynomials and the modified Macdonald polynomials are
related in the same way that the cardinalities of the Lusztig varieties and the parabolic Springer fibers are related.
This relationship can be elegantly expressed as a product of two matrices, the expansion matrix and the contraction matrix.
The consequence is that the plethystic transformation which converts integral form Macdonald polynomials
to modified Macdonald polynomials is captured by the expansion-contraction product.

In section 5 we review the definitions of affine Lusztig varieties and parabolic affine Springer fibers and the connection
between modified Macdonald polynomials and counting points in nilpotent parabolic affine Springer fibers given by 
Mellit \cite[\S 5]{Me17}.  This modified Macdonald
to parabolic Springer relationship provides a connection between the monomial expansion of 
integral form Macdonald polynomials and counting points of affine Lusztig varieties.

\subsection{The path to generalization to Lie types other than $GL_n$}\label{toNG}

The expansion-contraction arises very naturally in the setting of the Hecke algebra $H$ where
\begin{enumerate}[topsep=0.1em,itemsep=-0.2em]
\item[(a)] Expansion describes the Geck-Rouquier basis (of the center of $H$)
written in terms of the standard basis of $H$,
\item[(b)] Contraction describes the projection onto a parabolic subalgebra of the Hecke algebra.
\end{enumerate}
The argument given in this paper (Corollary 2.3 and Proposition 2.4) 
relating the cardinalities of Lusztig varieties and parabolic Springer fibers
by expansion-contraction works exactly the same for general Lie types.

The relationship between symmetric functions and the center of $H$ is given
by Wan and Wang in \cite[Theorem 4.6]{WW12}.
It is an isomorphism that takes
\begin{enumerate}[topsep=0.1em,itemsep=-0.2em]
\item[(a)] the monomial symmetric function $m_\nu(x)$ to the Geck-Rouquier basis element $\kappa_\nu$, and
\item[(b)] the big Schur function $S_\lambda(x;t)$ to the minimal central idempotent $z^H_\lambda$.
\end{enumerate}
The integral form Macdonald polynomial $J_\mu(x;q,t)$ maps to the element of the Hecke algebra
$$A_\mu = \sum_{\nu} \Card(Y^{-1}_{I \gamma_\nu I}(u_\mu)) \kappa_\nu,
\qquad\hbox{where}\qquad \hbox{$Y^{-1}_{I\gamma_\nu I}(u_\mu)$
is a unipotent affine Lusztig variety.}
$$
Then expansion-contraction transforms $A_\mu$ to a ``modified Macdonald polynomial'', 
as an element of the Hecke algebra, in the form
$$M_\mu = \sum_{\pi} \Card(Y^{-1}_{I_\pi}(u_\mu)) 1_{P_\pi},
\qquad\hbox{where}\qquad \hbox{$Y^{-1}_{I_\pi}(u_\mu)$
is a parabolic affine Springer fiber,}
$$
and $1_{P_\pi}$ is the idempotent in the parabolic Hecke algebra corresponding to the trivial representation
(see the proof of Proposition \ref{pipartoLus}).
The point is that integral form Macdonald polynomials and modified Macdonald polynomials 
sit naturally as \emph{elements of the Hecke algebra} 
and capture the cardinalities and relationship between Lusztig varieties and parabolic Springer fibers.

The only part of this picture that does not extend to Lie types other than type $GL_n$ is the connection between
symmetric functions and elements of the center of the Hecke algebra.  Provided one is willing to consider integral
form Macdonald polynomials and modified Macdonald polynomials as elements $A_\mu$ and $M_\mu$ in the Hecke algebra
then these objects generalize to all Lie types and serve as useful generating functions for enumerating the
points of affine Lusztig varieties and parabolic affine Springer fibers.

\subsection{Acknowledgments}

I would like to thank Persi Diaconis for drawing me into the work which resulted in the paper \cite{DRS22}.  It was this work that 
led me to revisit \cite{HR99} at an opportune moment.  I am very grateful to Xuhua He for educating me about Lusztig varieties.
His paper \cite{He23} and discussions with him during his recent visit to University of Melbourne were very helpful for me.

Finally, it is a pleasure to dedicate this paper to Peter Littelmann who has been a friend and inspiration for a long time.

\section{Hecke algebras and $\mathbf{1}_B^G$}

\subsection{The $(G,H)$-bimodule $\mathbf{1}_B^G$}

Let $\FF_q$ be a finite field with $q$ elements.
Let $G=GL_n(\FF_q)$, let $B$ be the subgroup of upper triangular matrices and let $W$
be the subgroup of permutation matrices (the symmetric group $S_n$).  The Bruhat decomposition is
the double coset decomposition
$$G = \bigsqcup_{w\in W} BwB.$$
(More generally, one can let $G$ be a finite group with a BN-pair and let $W$ be the Weyl group, see
\cite[Ch.\ IV \S2]{Bou}.)

Let $\CC G = \hbox{$\CC$-span}\{ g\in G\}$ be the group algebra of $G$.
The Hecke algebra $H$ is the subalgebra of $\CC G$ given by
$$H = \hbox{$\CC$-span}\{ T_w\ |\ w\in W\},
\qquad\hbox{where}\qquad
T_w = \frac{1}{\vert B \vert} \sum_{x\in BwB} x.
$$
(The Hecke algebra $H$ has unit $T_1$; it is a nonunital subalgebra of $\CC G$.)
Let $w\in W$ and let $\ell(w)$ denote the length of $w$ and let $s_i$ be a simple reflection in $W$.
Then (see \cite[(67.2)]{CR81} and/or \cite[Ch. IV \S2, Ex.\ 24]{Bou})
$$\frac{\vert B \vert}{\vert BwB\vert} = q^{-\ell(w)}
\qquad\hbox{and}\qquad
T_{s_i} T_w = \begin{cases}
(q-1)T_w + q T_{s_iw}, &\hbox{if $\ell(s_iw)<\ell(w)$,} \\
T_{s_i w}, &\hbox{if $\ell(s_iw)>\ell(w)$,}
\end{cases}
$$
and
\begin{equation}
T_wT_{s_i}  = \begin{cases}
(q-1)T_w + q T_{ws_i}, &\hbox{if $\ell(ws_i)<\ell(w)$,} \\
T_{ws_i}, &\hbox{if $\ell(ws_i)>\ell(w)$.}
\end{cases}
\label{Hmult}
\end{equation}

For $g\in G$ let
\begin{equation}
v_g = \frac{1}{\vert B \vert} \sum_{x\in gB} x,
\qquad\hbox{and define}\qquad
\mathbf{1}_B^G = \hbox{$\CC$-span}\{ v_g\ |\ g\in G\}.
\label{1BGdefn}
\end{equation}
Then $\dim(\mathbf{1}_B^G) = \vert G/B\vert$.
$$\hbox{$\CC G$ acts by left multiplication on $\mathbf{1}_B^G$}
\qquad\hbox{and}\qquad
\hbox{$H$ acts by right multiplication on $\mathbf{1}_B^G$.}
$$
As a $G$-module $\mathbf{1}_B^G$ is isomorphic to the trivial representation of $B$ induced to $G$,
and $H$ is the centralizer algebra,
$$\mathbf{1}_B^G\cong \Ind_B^G(triv)
\qquad\hbox{and}\qquad
H \cong \End_G(\mathbf{1}_B^G).
$$
Let $\hat H$ be an index set for the irreducible $H$-modules.
By the centralizer theorem (see \cite[Theorem 5.4]{HR04}), as $(G,H)$-bimodules,
\begin{equation}
\mathbf{1}_B^G \cong \bigoplus_{\lambda\in \hat H} G^\lambda \otimes H^\lambda,
\label{1BGdecomp}
\end{equation}
where the sum is over an index set for irreducible representations of $H$, $G^\lambda$ is an irreducible $G$-module indexed
by $\lambda$ and $H^\lambda$ is the irreducible $H$-module indexed by $\lambda$.
The $G^\lambda$ are the \emph{unipotent representations} of $G$.

The irreducible characters of $H$ are the functions
$$\chi^\lambda_H\colon H\to \CC\qquad\hbox{given by}\qquad \chi^\lambda_H(T_w) = \tr(T_w, H^\lambda).$$
The irreducible unipotent characters of $G$ are the functions
$$\chi^\lambda_G\colon \CC G\to \CC\qquad\hbox{given by}\qquad \chi^\lambda_G(g) = \tr(g, G^\lambda).$$
The decomposition in \eqref{1BGdecomp} gives that if $g\in G$ and $w\in W$ then
\begin{equation}
\tr(gT_w,\mathbf{1}_B^G) = \sum_{\lambda\in \hat H} \chi^\lambda_G(g)\chi^\lambda_H(T_w).
\label{btrcc}
\end{equation}

Define an inner product $\langle,\rangle_H\colon H\otimes H \to \CC$ by
$$\langle h_1, h_2\rangle_H = \tr(h_1h_2, \mathbf{1}_B^G), \qquad\hbox{for $h_1,h_2\in H$.}
$$
The basis
$$\Big\{ 
q^{-\ell(w)}T_{w^{-1}}\ |\ w\in W\Big\}
\qquad\hbox{is the dual basis to}\quad \{T_w\ |\ w\in W\},$$
with respect to $\langle , \rangle_H$ (see \cite[(11.30)(iii)]{CR81}).

\subsection{Bases of the center of $H$}

Since $\CC G$ is a semisimple algebra, then $\mathbf{1}_B^G$ is a semisimple $G$-module
and $H = \End_G(\mathbf{1}_B^G)$ is a semisimple algebra.  The center of $H$ is
$$Z(H) = \{ z\in H\ |\ \hbox{if $h\in H$ then $zh=hz$}\}.$$
As in \eqref{1BGdecomp}, let $\hat H$ be an index set for the irreducible $H$-modules.
The \emph{minimal idempotent} basis of $Z(H)$ (see \cite[(68.29)]{CR81} or \cite[(1.6)]{HLR}) is
\begin{equation}
\{ z_\lambda^H\ |\ \lambda\in \hat H\},
\qquad\hbox{where}\quad
z_\lambda^H = \frac{\chi^\lambda_G(1)}{\vert G/B\vert} \sum_{w\in W} \chi^\lambda_H(T_w)q^{-\ell(w)}T_{w^{-1}}.
\label{mincentH}
\end{equation}
Let $\cW$ be an index set for the conjugacy classes of $W$.  The \emph{conjugacy class basis}, 
or \emph{Geck-Rouquier basis}, of $Z(H)$ (see \cite[Cor.\ 8.2.4]{GP00}) is  
\begin{equation}
\{ \kappa_\nu\ |\ \nu\in \cW\}
\quad\hbox{given by}\qquad
\kappa_\nu = \sum_{w\in W} \kappa_{\nu,w} q^{-\ell(w)}T_w
\qquad\hbox{satisfying}\qquad 
\kappa_{\nu, \gamma_\mu} = \delta_{\nu\mu},
\label{GRbasis}
\end{equation}
whenever $\gamma_\mu$ is a minimal length of the conjugacy class $\cW_\mu$.  The condition $\kappa_{\nu,\gamma_\mu} = \delta_{\nu\mu}$ determines $\kappa_{\nu,w}$ for $w$ that are minimal length in their conjugacy class in $W$ and the remaining
$\kappa_{\nu,w}$ are forced by the condition that $\kappa\in Z(H)$.
The transition matrix between
the minimal central idempotents and the conjugacy class basis is the character table of the Hecke algebra,
$$z_\lambda^H = \sum_{\mu\in \cW} \chi^\lambda_H(T_{\gamma_\nu}) \kappa_\nu,
\qquad\hbox{where}\qquad
\begin{array}{l}
\hbox{$\gamma_\mu$ is a minimal length element}
\\
\hbox{ in the conjugacy class $\cW_\mu$.}
\end{array}
$$

\subsection{Conjugacy classes, Schubert cells and Lusztig varieties}

Let $\cC$ be an index set for the conjugacy classes of $G$.
$$G = \bigsqcup_{\mu\in \cC} \cC_\mu
\qquad\hbox{and}\qquad
G = \bigsqcup_{w\in W} BwB.
$$
For $g\in G$ and $w\in W$, 
$$
\hbox{the \emph{Lusztig variety} is}
\qquad
Y_{BwB}^{-1}(g) = \{ yB\in G/B\ |\ y^{-1}gy\in BwB\}.
$$
The following proposition shows that the number of points of $Y_{BwB}^{-1}(g)$  is related to the size of the intersection of the
conjugacy class of $g$ with the Schubert cell $Bw^{-1}B$.

\begin{prop}  \label{btrtoLu}
Let $g\in G$ and let $\cC_g$ be the conjugacy class of $g$.
\item[(a)]  Let $w\in W$.  Then
\begin{align*}
\Card(Y^{-1}_{BwB}(g^{-1}))
&=  \tr(gT_w, \mathbf{1}_B^G) 
= \frac{\vert G/B\vert}{\vert \cC_g\vert} \Card(\cC_g\cap Bw^{-1}B) 
= \sum_{\lambda\in \hat{H}} \chi^\lambda_G(g)\chi^\lambda_H(T_w)
\end{align*}
\item[(b)] Using notations as in \eqref{mincentH} and \eqref{GRbasis},
the element of $Z(H)$ given by
$$A_g =  \sum_{w\in W} \tr(gT_w, \mathbf{1}_B^G) q^{-\ell(w)} T_{w^{-1}}
= \sum_{\nu\in \cW} \tr(gT_{\gamma_\nu}, \mathbf{1}_B^G) \kappa_\nu
= \vert G/B\vert \sum_{\lambda \in \hat{H}} \frac{\chi^\lambda_G(g)}{\chi^\lambda_G(1)} z^H_\lambda
$$
acts on $\mathbf{1}_B^G$ the same way as the element of $Z(\CC G)$ given by $\frac{\vert G/B\vert}{\vert \cC_g\vert} C_g$, where
$$C_g = \sum_{x\in \cC_g} x.
$$
\end{prop}
\begin{proof}
By definition $T_1 = \frac{1}{\vert B\vert} \sum_{x\in B} x$.  
If $b\in B$ then $b T_1 = T_1 b = T_1$ and $T_1^2 = T_1$.  
If $g\in BwB$ and $g = b_1wb_2$ with $b_1,b_2\in B$ then
\begin{align}
T_1 g T_1 
&= T_1 b_1wb_2 T_1 
= T_1 w T_1 
= \frac{1}{\vert B\vert^2} \sum_{b_1, b_2\in B} b_1wb_2
\nonumber  \\
&= \frac{1}{\vert B\vert^2} \frac{\vert B\vert^2}{\vert BwB\vert} \sum_{x\in BwB} x
= \frac{\vert B\vert }{\vert BwB\vert} \frac{1}{\vert B\vert} \sum_{x\in BwB} x
= \frac{\vert B\vert }{\vert BwB\vert} T_w  = q^{-\ell(w)}T_w.
\label{gtoBwB}
\end{align} 

\smallskip\noindent
(a)  
Let $z_\lambda^G$ be the minimal central idempotent in
$Z(\CC G)$ which acts on $G^\lambda$ by the identity.
Then $\frac{\vert G/B\vert}{\vert \cC_g\vert} C_g$ acts on $\mathbf{1}_B^G$ the same way as
$$\frac{\vert G/B\vert}{\vert \cC_g\vert} \sum_{\lambda \in \hat{H}} \frac{\chi^\lambda_G(C_g)}{\chi_G^\lambda(1)}  z^G_\lambda
=\frac{\vert G/B\vert}{\vert \cC_g\vert} \sum_{\lambda \in \hat{H}} \vert \cC_g\vert \frac{\chi^\lambda_G(g)}{\chi_G^\lambda(1)}  z^G_\lambda
$$
and the same way as 
\begin{align}
A_g &= \frac{\vert G/B\vert}{\vert \cC_g\vert} 
\sum_{\lambda \in \hat{H}} \frac{\vert \cC_g \vert}{\chi^\lambda_G(1)} \chi^\lambda_G(g) z^H_\lambda
\nonumber \\
&=\vert G/B\vert 
\sum_{\lambda \in \hat{H}} \frac{\chi^\lambda_G(g)}{\chi^\lambda_G(1)}  
\frac{1}{\vert G/B\vert} \sum_{w\in W} \chi^\lambda_G(1) \chi^\lambda_H(T_w) q^{-\ell(w)}T_{w^{-1}}
\nonumber  \\
&= \sum_{w\in W}
\Big( \sum_{\lambda\in \hat{H}} \chi^\lambda_G(g)\chi^\lambda_H(T_w)\Big) q^{-\ell(w)}T_{w^{-1}}
\nonumber \\
&= \sum_{w\in W}  \tr(g T_w, \mathbf{1}_B^G) q^{-\ell(w)}T_{w^{-1}},
\label{Dgtobtr}
\end{align}
where the second equality uses \eqref{mincentH} and the last equality follows from \eqref{btrcc}.

Using that  $C_g$ is central in $\CC G$  gives $C_g T_1 = C_g T_1^2 = T_1 C_g T_1$ so that
\begin{align}
A_g = \frac{\vert G/B\vert}{\vert \cC_g\vert}C_g T_1 &= 
\frac{\vert G/B\vert}{\vert \cC_g\vert}\sum_{x\in \cC_g} T_1 x T_1
= \sum_{w\in W} \frac{\vert G/B\vert}{\vert \cC_g\vert}\Card(\cC_g \cap Bw^{-1}B) q^{-\ell(w)} T_{w^{-1}},
\label{DgfromCg}
\end{align}
where the last equality follows from  \eqref{gtoBwB}.
Comparing coefficients of $T_w$  in \eqref{Dgtobtr} and \eqref{DgfromCg} gives
$$\frac{\vert G/B\vert}{\vert \cC_g\vert}\Card(\cC_g \cap Bw^{-1}B) =  \tr(gT_w,\mathbf{1}_B^G)
= 
\sum_{\lambda\in \hat{H}} \chi^\lambda_G(g)\chi^\lambda_H(T_w)\,.
$$
For $h\in \mathbf{1}_B^G$, let
$h\vert_{v_g}$ denote the coefficient of $v_g$ when $h$ is expanded
in the basis $\{ v_g\ |\ gB\in G/B\}$ given in  \eqref{1BGdefn}.  Then
\begin{align*}
\tr(&gT_w, \mathbf{1}_B^G) 
= \sum_{yB\in G/B} gv_yT_w \big\vert_{v_y} = \sum_{yB\in G/B} v_y T_w \big\vert_{v_{g^{-1}y}}
= \sum_{yB\in G/B} \sum_{zB\in G/B\atop zB\in yBwB} v_z\big\vert_{v_{g^{-1}y}} \\
&= \#\{yB\in G/B\ |\ g^{-1}yB\in yBwB\}
= \#\{yB\in G/B\ |\ y^{-1}g^{-1}y\in BwB\} = \#Y^{-1}_{BwB}(g^{-1}).
\end{align*}
\end{proof}

A \emph{trace on $H$} is a linear function $\chi\colon H\to \CC$ such that 
$$\hbox{if $h_1, h_2\in H$\quad then\quad 
$\chi(h_1h_2) = \chi(h_2h_1)$.}
$$
Using \eqref{Hmult}, 
if $\chi\colon H\to \CC$ is a trace then
$$\chi(T_{s_i w})=\chi(T_{s_i}T_w) =\chi(T_wT_{s_i})=\chi(T_{w s_i}),
\qquad\hbox{if $\ell(s_iws_i)=\ell(w)$ and $\ell(s_i w) = \ell(w)+1$. and}
$$
$$\chi(T_{s_iws_i})=\chi(T_{s_i}T_wT_{s_i})=\chi(T^2_{s_i}T_w)
=q \chi(T_w)+(q-1)\chi(T_{s_i w}),
\qquad\hbox{if $\ell(s_iws_i) = \ell(w)+2$.}$$
Since $\Card(Y^{-1}_{BwB}(g)) = \tr(g^{-1}T_w, \mathbf{1}_B^G)$ is a trace (as a function on $H$) 
then these trace relations imply the following Corollary.
Corollary \ref{LVtrrels} is also a consequence of the slightly more refined statement about the structure of Lusztig varieties
stated in [He23, §4.4] (with a reference to \cite[proof of Theorem 1.6]{DL76} for the proof).

\begin{cor} \label{LVtrrels}
Let $g\in G$.  Let $w\in W$ and $i\in \{1, \ldots, n-1\}$ such that $\ell(s_iw) = \ell(w)+1$.  Then
\begin{align*}
\Card(Y^{-1}_{Bs_iwB}(g)) &= \Card(Y^{-1}_{Bws_iB}(g)), &&\hbox{if $\ell(s_iw)=\ell(ws_i)$, and} \\
\Card(Y^{-1}_{Bs_iws_iB}(g)) &= q\Card(Y^{-1}_{Bs_iwB}(g))+ (q-1)\Card(Y^{-1}_{BwB}(g)), &&\hbox{if $\ell(s_iws_i)=\ell(w)+2$.}
\end{align*}
\end{cor}

\subsection{Parabolic Springer fibers} \label{PSpr}

Let $\pi = (\pi_1,\ldots,\pi_\ell)$ with $\pi_1,\ldots, \pi_\ell\in \ZZ_{>0}$ and $\pi_1+\cdots+\pi_\ell=n$.
Let $P_\pi$ be the parabolic subgroup of $G$ consisting of block upper triangular matrices with
block sizes $\pi_1,\ldots, \pi_\ell$.    Let $W_\pi$ be the subgroup of $W= S_n$ given by
$$W_\pi = S_{\pi_1}\times \cdots \times S_{\pi_\ell}
\qquad\hbox{so that}\qquad
P_\pi = \bigsqcup_{w\in W_\pi} BwB.
$$
Let
$$W_\pi(q) = \sum_{w\in W_\pi} q^{\ell(w)}
\qquad\hbox{so that}\qquad
\Card(P_\pi/B) = W_\pi(q).
$$

Let $g\in G$.  The \emph{$\pi$-parabolic Springer fiber over $g$} is
\begin{equation}
Y^{-1}_{P_\pi}(g) = \#\{yP_\pi\in G/P_\pi\ |\ y^{-1}gy\in P_\pi\}.
\label{Sfdefn}
\end{equation}
The following Proposition counts the number of points of $Y^{-1}_{P_\pi}(g)$ in terms of the
sizes of the Lusztig varieties.

\begin{prop} \label{pipartoLus}
$$
\Card(Y^{-1}_{P_\pi}(g)) 
= \frac{1}{W_\pi(q)}\sum_{w\in W_\pi} \Card(Y^{-1}_{BwB}(g)). 
$$
\end{prop}
\begin{proof}
Let
$$
W_\pi(q) = \sum_{w\in W_\pi} q^{\ell(w)}
\qquad\hbox{and let}\qquad
1_{P_\pi} = \frac{1}{W_\pi(q)} \sum_{w\in W_\pi} T_w,
$$
which is an idempotent in the Hecke algebra $H$.
As a $G$-module
$$\mathbf{1}_{P_\pi}^G = \mathbf{1}_B^G\cdot 1_{P_\pi}\cong \Ind_{P_\pi}^G(\mathrm{triv}),$$
the trivial representation of $P_\pi$ induced to $G$.  
Then the number of points of $Y^{-1}_{P_\pi}(g)$ is given by a trace:
\begin{align*}
\Card(Y^{-1}_{P_\pi}(g)) 
&= \#\{yP_\pi\ |\ gyP_\pi = yP_\pi\}
= \tr(g, \mathbf{1}_{P_\pi}^G) 
= \tr(g, \mathbf{1}_B^G\cdot 1_{P_\pi})  \\
&= \frac{1}{W_\pi(q)} \sum_{w\in W_\pi} \tr(gT_w, \mathbf{1}_B^G) 
= \frac{1}{W_\pi(q)}\sum_{w\in W_\pi} \Card(Y^{-1}_{BwB}(g)). 
\end{align*}
\end{proof}


\section{Counting in $GL_n(\FF_q)$}

\subsection{Macdonald polynomials}

Fix $n\in \ZZ_{>0}$.
A \emph{partition of $n$} is a sequence
$\lambda = (\lambda_1,\ldots, \lambda_\ell)$ of positive integers with $\lambda_1\ge \ldots\ge \lambda_\ell >0$
and $\lambda_1+\cdots+\lambda_n$.  For $\lambda = (\lambda_1, \ldots, \lambda_n)$, define
$$\ell(\lambda)= \ell \qquad
\qquad\hbox{and}\qquad
n(\lambda) = \sum_{i=1}^\ell (i-1)\lambda_i.
$$
Let 
\begin{enumerate}[itemsep=0em]
\item[] $J_\mu(x;q,t)$ be the integral form Macdonald polynomials \cite[Ch. VI (8.3)]{Mac},
\item[] $S_\lambda(x;t)$ the Big Schur functions \cite[Ch.\ III (4.5)]{Mac} and 
\item[] $m_\nu$ the monomial symmetric functions.
\end{enumerate}
Define $K_{\lambda\nu}(q,t)$, $a_{\mu\nu}(q,t)$  and $L_{\nu\lambda}(t)$ by
$$
J_\mu(x;q,t) = \sum_\lambda K_{\lambda\mu}(q,t) S_\lambda(x;t),
\qquad\hbox{and}\qquad
J_\mu(x;q,t) = \sum_\nu a_{\mu\nu}(q,t) (1-t)^{\ell(\nu)} m_\nu,
$$
and
\begin{equation}
S_\lambda(x;t) = \sum_{\nu} L_{\lambda\nu}(t) (1-t)^{\ell(\nu)} m_\nu,
\qquad\hbox{so that}\qquad
a_{\mu\nu}(q,t) = \sum_\lambda K_{\lambda\mu}(q,t) L_{\lambda\nu}(t).
\label{aeqkL}
\end{equation}
The Schur functions $s_\lambda$ and the modified Macdonald polynomials $\widetilde{H}_\mu[X;q,t]$ are given by
\begin{equation}
s_\lambda = \sum_\pi K_{\lambda\pi}(0,1) m_\pi.
\label{modMac}
\end{equation}
$$
\begin{array}{l}
\displaystyle{
\widetilde{H}_\mu[X;q,t] 
= \sum_\lambda t^{n(\lambda)}K_{\lambda\mu}(q,t^{-1})s_\lambda,
} \\
\displaystyle{
\widetilde{H}_\mu[X;q,t] 
= \sum_\pi b_{\mu\pi}(q,t) m_\pi,
}
\end{array}
\qquad\hbox{so that}\qquad
b_{\mu\nu}(q,t) = \sum_\lambda t^{n(\lambda)}K_{\lambda\mu}(q,t^{-1}) K_{\lambda\pi}(0,1).
$$
We follow \cite[Ch.\ VI]{Mac} in using nonplethystic notation for the variables in $J_\mu(x;q,t)$ and \cite{GH96} and \cite{Me17}
in using plethystic notation for the variables in $\widetilde{H}_\mu[X;q,t]$.

Another way to state the relation between $\widetilde{H}_\mu$ and $J_\mu$ is via the plethystic transformation
which has the effect of changing $S_\lambda(x;t)$ to $s_\lambda(x)$. 
As in \cite[(11)]{GH96},
\begin{equation}
\widetilde{H}_\lambda[X;q,t] = t^{n(\lambda)} J_\lambda\Big[\frac{X}{1-t^{-1}}; q,t^{-1} \Big].
\label{MacPl}
\end{equation}
From \cite[Cor.\ 3.2]{GH96} or \cite[(8.14) and (8.15)]{Mac},
\begin{equation}
\widetilde{H}_\lambda[X;q,t] = \widetilde{H}_{\lambda'}[X;t,q],
\label{modMaconj}
\end{equation}
where $\lambda'$ is the conjugate partition to $\lambda$.

\subsection{Counting points of Lusztig varieties in $GL_n(\FF_q)$}

Let $\FF_q$ be a finite field with $q$ elements.
Let $B$ be the subgroup of upper triangular matrices in $G=GL_n(\FF_q)$ and let
$H$ be the Hecke algebra for $B\subseteq G$.

 Let $\lambda,\mu, \nu$ be partitions of $n$.  
  Let $\chi^\lambda_G$
be the character of the irreducible unipotent $GL_n(\FF_q)$-representation $G^\lambda$
and let $u_\mu$ be a unipotent element in $GL_n(\FF_q)$ with Jordan form corresponding to $\mu$.
From \cite[(2.2)]{Lu81} or \cite[Theorem 4.9(c)]{HR99},
\begin{equation}
\chi^\lambda_G(u_\mu) = q^{n(\mu)} K_{\lambda\mu}(0,q^{-1}).
\label{unipchar}
\end{equation}
This identity provides a representation theoretic viewpoint on (a specialization) of the coefficients $K_{\lambda\mu}(q,t)$ which appear
in \eqref{aeqkL}. 

With respect to the  inner product $\langle , \rangle_{0,t}$ of \cite[Ch.\ III]{Mac}, the Schur functions $s_\lambda$ and the
Big Schur functions $S_\lambda = S_\lambda(x;t)$ are dual bases.  
The dual basis to the monomial symmetric functions $m_\mu$ is denoted
$q_\mu = q_\mu(x;t)$ in \cite[(4.8) and (4.10)]{Mac}.  In formulas,
$\langle s_\lambda, S_\mu\rangle_{0,t} = \delta_{\lambda\mu}$ and
$\langle q_\nu, m_\mu\rangle_{0,t} = \delta_{\nu\mu}.$
Thus
$$L_{\lambda\nu}(t)(1-t)^{\ell(\nu)} = \langle q_\nu, S_\lambda\rangle_{0,t}
\qquad\hbox{which gives}\qquad
\frac{1}{(1-t)^{\ell(\nu)}} q_\nu(x;t) = \sum_\lambda L_{\lambda\nu}(t) s_\lambda(x).
$$
Let $\lambda, \nu$ be partitions of $n$.  
Let $\chi^\lambda_H$ be the character of the irreducible $H$-representation $H^\lambda$. 
Let $\gamma_\nu$ be a  minimal length element of the conjugacy class in $W$ of 
permutations of cycle type $\nu$.
By \cite[Th.\ 4.14]{Ra91},
$$
\frac{q^{n-\ell(\nu)} }{(q-1)^{\ell(\nu)} } q_\nu(x;q^{-1}) 
= \sum_\lambda \chi^\lambda_H(T_{\gamma_\nu^{-1}}) s_\lambda(x)
$$
so that
\begin{equation}
\chi^\lambda_H(T_{\gamma_\nu^{-1}}) = 
q^{n-\ell(\nu)} L_{\lambda\nu}(q^{-1}).
\label{Hchar}
\end{equation}
This identity provides a representation theoretic viewpoint on the coefficients $L_{\lambda\nu}(t)$ which appear
in \eqref{aeqkL}. 

The following Theorem is a reformulation of the first displayed equation in the proof of \cite[Theorem 4.11]{HR99}.
It provides a geometric viewpoint for a specialization of the coefficients $a_{\mu\nu}(q,t)$ which appear
in \eqref{aeqkL}. 

\begin{thm} \label{HRLvcount}
Let $\mu, \nu$ be partitions of $n$.  
Let $u_\mu$ be a unipotent element in $GL_n(\FF_q)$ with Jordan form corresponding to $\mu$ and let
$\gamma_\nu$ be a  minimal length element of the conjugacy class in $W$ of 
permutations of cycle type $\nu$.
Then
$$
q^{n(\mu)+n-\ell(\nu)}a_{\mu\nu}(0,q^{-1})
= \Card(Y^{-1}_{B\gamma_\nu B}(u_\mu)),
$$
the number of points of the Lusztig variety $Y^{-1}_{B\gamma_\nu B}(u_\mu)$ over $\FF_q$.  
\end{thm}
\begin{proof}
By Proposition \ref{btrtoLu}, \eqref{unipchar}, \eqref{Hchar} and the last identity in \eqref{aeqkL},
\begin{align*}
\Card(Y^{-1}_{B\gamma_\nu B}(u_\mu)) 
&= \sum_\lambda \chi^\lambda_G(u_\mu) \chi^\lambda_H(T_{\gamma_\nu^{-1}})  
= \sum_\lambda q^{n(\mu)} K_{\lambda\mu}(0,q^{-1}) L_{\lambda\nu}(q^{-1})q^{n-\ell(\nu)} \\
&= q^{n(\mu)+n-\ell(\nu)} a_{\mu\nu}(0,q^{-1}),
\end{align*}
\end{proof}

\subsection{Modified Hall-Littlewoods and parabolic Springer fibers}

Recall from \eqref{aeqkL} and \eqref{modMac} that the $a_{\mu\nu}(q,t)$ 
give the expansion of integral form Macdonald polynomials in monomial symmetric functions
and the $b_{\mu\pi}(q,t)$ specify the transition matrix between the modified Macdonald polynomials and
the monomial symmetric functions,
$$
J_\mu(x;q,t) = \sum_\nu a_{\mu\nu}(q,t) (1-t)^{\ell(\nu)} m_\nu
\qquad\hbox{and}\qquad
\widetilde{H}_\mu[X;q,t] 
= \sum_\pi b_{\mu\pi}(q,t) m_\pi.
$$
The corresponding Hall-Littlewood polynomials are the specializations  $J_\mu(x;0,t)$ and $\widetilde{H}_\mu[X;0,t]$.

An elegant proof of the following result is found in \cite[Theorem 2.12 and Cor.\ 2.13]{Me17}.
Alternative references are \cite[(9)]{LLT95} and \cite[Cor. 8.7]{HS79}.
Although the result can be deduced fairly easily from identities in \cite{Mac} (see \cite[(9)]{LLT95}, \cite[Theorem 4.9(a)]{HR99}, 
\cite[Ch.\ IV \S4]{Mac} and \cite[Ch.\ IV \S 4 Ex. 1]{Mac}) getting the necessary variable specializations and conjugate partitions organized  cleanly is painful enough that it is convenient
to stick with the clean, direct, proof found in \cite[Theorem 2.12 and Cor.\ 2.13]{Me17}.

\begin{thm} \label{Sftob}
Let $\mu$ and $\pi$ be partitions of $n$.  Let $u_\mu\in GL_n(\FF_q)$ 
be a unipotent element in Jordan normal form with Jordan block sizes given by the partition $\mu$
and let $P_\pi\subseteq GL_n(\FF_q)$ be the standard parabolic subgroup of block upper triangular matrices 
with blocks sizes determined
by the partition $\pi$.  Let $Y^{-1}_{P_\pi}(u_\mu)$ be the $\pi$-parabolic Springer fiber over $u_\mu$ defined in 
\eqref{Sfdefn}.  Then
$$b_{\mu\pi}(0,q) = \Card(Y^{-1}_{P_\pi}(u_\mu^{-1})).
$$
\end{thm}

\section{Expansion-contraction}

Let $n\in \ZZ_{>0}$. As in \eqref{GRbasis}, let $\cW$ be an index set for the conjugacy classes of $W$ (for $W=S_n$ the set
$\cW$ is naturally identified with the set of partitions of $n$).

The \emph{expansion matrix} is $\kappa = (\kappa_{\nu,w})$ (rows indexed by $\mu\in \cW$ and columns indexed by $w\in W$)
given by
\begin{enumerate}[itemsep=-0.2em]
\item[(a)]  $\kappa_{\nu,\gamma_\mu}=\delta_{\mu\nu}$, if $\gamma_\mu$ is minimal length in the conjugacy class $\cW_\mu$,
\item[(b)] $\kappa_{\nu, s_iw} = \kappa_{\nu, ws_i}$ if $\ell(s_iw) = \ell(ws_i)$ and $\ell(s_iw) = \ell(w)+1$,
\item[(c)] $\kappa_{\nu, s_iws_i} = q\kappa_{\nu, w} + (q-1)\kappa_{\nu,ws_i}$ if $\ell(s_iws_i) = \ell(w)+2$.
\end{enumerate}
Algebraically, the matrix $\kappa$ contains the coefficient of the Geck-Rouquier basis of $Z(H)$ when expanded in terms of the
basis $\{q^{-\ell(w)} T_{w^{-1}}\ |\ w\in W\}$ of $H$.  Geometrically, the expansion matrix is related to the trace identities for
Lusztig varieties in Corollary \ref{LVtrrels}.

With notations as in Section \ref{PSpr}, the \emph{contraction matrix} is $C = (C_{w,\pi})$ given by 
$$C_{w,\pi} = \begin{cases}
1, &\hbox{if $w\in W_\pi$,} \\
0, &\hbox{if $w\not\in W_\pi$,} 
\end{cases}
\qquad\hbox{where}\qquad
W_\pi = S_{\pi_1}\times \cdots \times S_{\pi_\ell}
$$
is a parabolic subgroup of $W$ (a Young subgroup of $S_n$) corresponding to the partition $\pi = (\pi_1, \ldots, \pi_\ell)$.
Algebraically, the contraction matrix captures the idempotent $\mathbf{1}_\pi\in H$ that projects
$\mathbf{1}_B^G \to \mathbf{1}_{P_\pi}^G$. 
Geometrically, the contraction describes the relation between Lusztig varieties and parabolic Springer fibers that
appears in Proposition \ref{pipartoLus}.

\subsection{Expansion-contraction and Macdonald polynomials}

Recall from \eqref{aeqkL} and \eqref{modMac} that the $a_{\mu\nu}(q,t)$ 
give the expansion of integral form Macdonald polynomials in monomial symmetric functions
and the $b_{\mu\pi}(q,t)$ specify the transition matrix between the modified Macdonald polynomials and
the monomial symmetric functions,
$$
J_\mu(x;q,t) = \sum_\nu a_{\mu\nu}(q,t) (1-t)^{\ell(\nu)} m_\nu
\qquad\hbox{and}\qquad
\widetilde{H}_\mu[X;q,t] 
= \sum_\pi b_{\mu\pi}(q,t) m_\pi,
$$
For $k\in \ZZ_{>0}$ let
$$[k] = 1+q+\cdots+q^{k-1}, \quad [k]! = [k][k-1]\cdots [2][1],
\qquad\hbox{and let}\qquad
[\pi]! = [\pi_1]! \cdots [\pi_\ell]!.
$$
The following Theorem tells us that, up to normalizations, the $b_{\mu\pi}(t,q)$ are obtained from the $a_{\mu\nu}(t,q^{-1})$ by
multiplying the the expansion matrix and the contraction matrix.

\begin{thm} \label{atob}
Let $\mu,\pi$ be partitions of $n$.  Then
$$\sum_{\nu\vdash n}\sum_{w\in S_n} 
q^{n(\mu)}a_{\mu\nu}(t,q^{-1})q^{n-\ell(\nu)}\kappa_{\nu,w} C_{w,\pi}\frac{1} { [\pi]!} = b_{\mu\pi}(t,q).$$
\end{thm}
\begin{proof}
Let  $A(t,q^{-1})$, $b(t,q)$, $D$ and $M$ be the matrices given by
\begin{align*}
&A(t,q^{-1}) = (q^{n(\mu)}a_{\mu\nu}(t,q^{-1})), \qquad\qquad
b(t,q)= (b_{\mu\pi}(t,q)), \\
&\hbox{$D =\mathrm{diag}(d_\pi)$ is the diagonal matrix with diagonal entries $d_\pi = \frac{1}{[\pi]!}$, \quad and} \\
&\hbox{$M = \mathrm{diag}(q^{n-\ell(\nu)})$ is the diagonal matrix with diagonal entries $q^{n-\ell(\nu)}$.}
\end{align*}
Using the notations in \eqref{modMac} and \eqref{aeqkL}, define matrices 
$$K(0,1)= (K_{\lambda\mu}(0,1)), \qquad
F = K(0,1)^{-1} = (F_{\pi\lambda})
\qquad\hbox{and}\qquad 
L = (L_{\lambda\nu}(q^{-1}))$$
so that
$$m_\pi = \sum_{\lambda} F_{\pi\lambda} s_\lambda
\qquad\hbox{and}\qquad
s_\lambda[X(1-q^{-1})] = S_\lambda(q^{-1}) = \sum_{\nu} L_{\lambda\nu}(q^{-1})(1-q^{-1})m_\nu[X].
$$
Using the identities in \eqref{aeqkL}, \eqref{modMac} and \eqref{MacPl},
\begin{align*}
\sum_\nu q^{n(\mu)}a_{\mu\nu}(t,q^{-1})(1-q^{-1})^{\ell(\nu)} m_\nu[X] 
&= q^{n(\mu)}J_\mu[X;t,q^{-1}] = \widetilde{H}_\mu[X(1-q^{-1});t,q] \\
&= \sum_\pi b_{\mu\pi}(t,q)m_\pi[X(1-q^{-1})] \\
&= \sum_{\pi,\lambda} b_{\mu\pi}(t,q) F_{\pi\lambda} s_\lambda[X(1-q^{-1})] \\
&= \sum_{\pi,\lambda} b_{\mu\pi}(t,q) F_{\pi\lambda} L_{\lambda\nu}(q^{-1})(1-q^{-1})^{\ell(\nu)} m_\nu[X].
\end{align*}
Thus
\begin{equation}
A(t,q^{-1})=b(t,q) F L.
\label{AtoBFL}
\end{equation}

Since $\Card(P_\pi/B) = [\pi]!$ then
Theorem \ref{HRLvcount},
Corollary \ref{LVtrrels},
Proposition \ref{pipartoLus},
and Theorem \ref{Sftob} give
\begin{align*}
\sum_{\nu\vdash n}\sum_{w\in S_n} 
q^{n(\mu)}a_{\mu\nu}(0,q^{-1})q^{n-\ell(\nu)}\kappa_{\nu,w} C_{w,\pi}\frac{1} { [\pi]!} 
&=
\sum_{\nu\vdash n}\sum_{w\in S_n} 
\Card(Y^{-1}_{B\gamma_\nu B}(u_\mu)) \kappa_{\nu,w} C_{w,\pi}\frac{1} { [\pi]!}  \\
&=
\sum_{w\in S_n} 
\Card(Y^{-1}_{BwB}(u_\mu)) C_{w,\pi}\frac{1} { [\pi]!}  \\
&=  \Card(Y^{-1}_{P_\pi}(u_\mu))\Card(P_\pi/B) \frac{1} { [\pi]!}  \\
&=  \Card(Y^{-1}_{P_\pi}(u_\mu)) 
= b_{\mu\pi}(0,q),
\end{align*}
so that
$$A(0,q^{-1})M \kappa CD = b(0,q).$$
By \eqref{AtoBFL},
\begin{equation}
A(0,q^{-1}) = b(0,q)FL,
\qquad\hbox{which gives}\qquad
M \kappa C D = (FL)^{-1}.
\label{GLinv}
\end{equation}
Plugging this back into  \eqref{AtoBFL} gives
$A(q,t) M\kappa CD = b(t,q)$, which is equivalent to the statement in Theorem \ref{atob}.
\end{proof}

\subsection{Expansion-contraction and the plethystic transformation}

Let $R = (R_{\nu\pi})$ be the matrix with rows and columns indexed by partitions of $n$ given by
$$m_\nu\left[\frac{X}{1-q^{-1}}\right] = \sum_\pi R_{\nu\pi}(q) m_\pi[X].$$
The matrix $R$ is the change of basis matrix for the plethystic transformation, with respect to the monomial symmetric functions.
The following Corollary says that the matrix $R$ is, up to normalization of rows and columns, the expansion-contraction.
The matrix $R$ is a square matrix, but the expansion matrix $\kappa$ and the contraction matrix $C$ are not square.

\begin{cor} Let $\nu$ and $\pi$ be partitions of $n$.
Then
$$R_{\nu\pi}(q) = \sum_{w\in S_n} q^{n-\ell(\nu)} \kappa_{\nu,w}C_{w,\pi}\frac{1}{[\pi]!}.$$
\end{cor}
\begin{proof}
Use the matrix notations from the proof of Theorem \ref{atob}.
Since
\begin{align*}
m_\pi[X(1-q^{-1})] &= \sum_{\lambda} F_{\pi\lambda} s_\lambda[X(1-q^{-1})]
\quad\hbox{and}\quad
s_\lambda[X(1-q^{-1})] = \sum_\nu L_{\lambda\nu}(1-q^{-1})^{\ell(\nu)}m_\nu[X].
\end{align*}
then
$$R^{-1} = FL 
\qquad\hbox{so that}\qquad
R = (FL)^{-1} = M\kappa CD,
$$
by \eqref{GLinv}.
\end{proof}


\section{Counting points in affine Springer fibers and Lusztig varieties}

Let $B(\FF_q)$ be the subgroup of $GL_n(\FF_q)$ of upper triangular matrices and
for a partition $\pi = (\pi_1, \ldots, \pi_\ell)$ of $n$ let 
$$P_\pi(\FF_q) \subseteq GL_n(\FF_q)\qquad\hbox{be the subgroup of $GL_n(\FF_q)$ of}\qquad
\begin{array}{c}
\hbox{block upper triangular matrices} \\
\hbox{with block sizes $\pi_1, \ldots, \pi_\ell$.}
\end{array}
$$
Then $B(\FF_q)=P_{(1^n)}(\FF_q)$.
Define $G$, $K$, $I_\pi$ and $I$ by
$$
\begin{matrix}
\widetilde{G}=GL_n(\FF_q((\epsilon))) \\
\cup\vert \\
K = GL_n(\FF_q[[\epsilon]]) &\stackrel{\epsilon=0}\longrightarrow &GL_n(\FF_q) \\
\cup\vert &&\cup\vert
\\
I_\pi = \{ g\in K\ |\ g(0)\in P_\pi(\FF_q) \} &\stackrel{\epsilon=0}\longrightarrow &P_\pi(\FF_q) \\
\cup\vert &&\cup\vert
\\
I = \{ g\in K\ |\ g(0)\in B(\FF_q) \} &\stackrel{\epsilon=0}\longrightarrow &B(\FF_q) 
\end{matrix}
$$
so that $K = I_{(n)}$ and $I = I_{(1^n)}$.
The double coset decompositions for $\widetilde{G}\supseteq I$ and for $K\supseteq I$ are
$$\widetilde{G} = \bigsqcup_{w\in \widetilde{W}} IwI
\qquad\hbox{and}\qquad
K = \bigsqcup_{w\in W} IwI,
$$
where $\widetilde{W}$ is the affine Weyl group (the group of $n$-periodic permutations, see \cite[\S 5.1]{CR22})
and $W$ is the finite Weyl group (the symmetric group $S_n$).

Let $\mu$ be a partition of $n$.  Let
$u_\mu$ be a unipotent element in $GL_n(\FF_q)$ in Jordan normal form with Jordan block sizes given by the
partition $\mu$.  Since $GL_n(\FF_q)$ is a subgroup of $\widetilde{G} = GL_n(\FF_q((\epsilon))$ 
then $u_\mu$ is also an element of $G$.
The \emph{$\pi$-parabolic affine Springer fiber over $u_\mu$} is
\begin{equation}
Y_{I_\pi}^{-1}(u_\mu) = \{ yI_\pi\in G/I_\pi\ |\ y^{-1} u_\mu  y\in I_\pi\}.
\label{affSpr}
\end{equation}
Let $w\in W$.
The \emph{affine Lusztig variety for $w$ and $u_\mu$} is
\begin{equation}
Y^{-1}_{IwI}(u_\mu) = \{ yI\in G/I\ |\ y^{-1}u_\mu y\in IwI\}.
\label{affLusvar}
\end{equation}

In \cite[Theorem 5.15]{Me17},
Mellit gives precise meaning to $\Card(Y^{-1}_{I_\pi}(u_\mu))$ and shows that
$$b_{\mu\pi}(t,q) = \Card(Y^{-1}_{I_\pi}(u_\mu)).$$
We will not describe the weighted point count generating function that Mellit uses for  $\Card(Y^{-1}_{I_\pi}(u_\mu))$ 
as it is a bit too involved 
and would take us too far afield.

The point is that the affine Lusztig varieties $Y^{-1}_{IwI}(u_\mu)$ for $w\in W$
are related to the $\pi$-parabolic affine Springer fibers $Y^{-1}_{I_\pi}(u_\mu)$ considered in Mellit by expansion-contraction
in exactly the same way as the Lusztig varieties $Y^{-1}_{BwB}(u_\mu)$ are related to the $\pi$-parabolic Springer fibers
$Y^{-1}_{P_\pi}(u_\mu)$.  More precisely, affine versions of Proposition \ref{pipartoLus} and Corollary \ref{LVtrrels} 
hold (with exactly the same proof) as follows: if $g\in\widetilde{G}$ then
$$
\Card(Y^{-1}_{I_\pi}(g)) 
= \frac{1}{W_\pi(q)}\sum_{w\in W_\pi} \Card(Y^{-1}_{IwI}(g)),
$$
and
if $w\in W$ and $s_i$ is a simple reflection in $W$ such that
$\ell(s_iw)=\ell(w)+1$ then
\begin{align*}
\Card(Y^{-1}_{Is_iw I}(g)) &= \Card(Y^{-1}_{Iws_i I}(g)), &&\hbox{if $\ell(s_iw)=\ell(ws_i)$, and} \\
\Card(Y^{-1}_{I s_iws_i I}(g)) &= q\Card(Y^{-1}_{I s_iw I}(g))+ (q-1)\Card(Y^{-1}_{IwI}(g)), &&\hbox{if $\ell(s_iws_i)=\ell(w)+2$.}
\end{align*}
This means that we can use Theorem \ref{atob} and the computation 
(compare to the proof of Theorem \ref{atob})
\begin{align*}
\sum_{\nu\vdash n}\sum_{w\in S_n} 
q^{n(\mu)}a_{\mu\nu}(t,q^{-1})&q^{n-\ell(\nu)}\kappa_{\nu,w} C_{w,\pi}\frac{1} { [\pi]!} 
= b_{\mu\pi}(t,q) \\
&=  \Card(Y^{-1}_{I_\pi}(u_\mu)) 
=  \Card(Y^{-1}_{I_\pi}(u_\mu))\Card(I_\pi/P) \frac{1} { [\pi]!}  \\
&=
\sum_{w\in S_n} 
\Card(Y^{-1}_{IwI}(u_\mu)) C_{w,\pi}\frac{1} { [\pi]!}  \\
&=
\sum_{\nu\vdash n}\sum_{w\in S_n} 
\Card(Y^{-1}_{I\gamma_\nu I}(u_\mu)) \kappa_{\nu,w} C_{w,\pi}\frac{1} { [\pi]!}
\end{align*}
to conclude that 
\begin{equation}
\Card(Y^{-1}_{I\gamma_\nu I}(u_\mu)) = 
q^{n(\mu)}a_{\mu\nu}(t,q^{-1})q^{n-\ell(\nu)},
\label{mainpt}
\end{equation}
where $\Card(Y^{-1}_{I\gamma_\nu I}(u_\mu))$ is a weighted point count of exactly the same form as that used by Mellit.

\subsection{A concluding sticking point}\label{sticking}

In the framework of this paper, extending from the Hall-Littlewood case to the Macdonald polynomial case is achieved, geometrically,
by going to the ``affine case'' i.e., replacing $GL_n(\FF_q)$ by the loop group $GL_n(\FF_q((\epsilon))$.  In the afffine case,
the Lusztig variety $Y^{-1}_{I\gamma_\nu I}(u_\mu)$ 
and the parabolic Springer fiber $Y^{-1}_{I_\pi}(u_\mu)$ are infinite, necessitating a more refined definition of
$\Card(Y^{-1}_{I\gamma_\nu I}(u_\mu))$ and $\Card(Y^{-1}_{I_\pi}(u_\mu))$.  An attractive proposal is to define these
cardinalities as length generating functions over the affine Weyl group,
\begin{align}
\Card(Y^{-1}_{I\gamma_\nu I}(u_\mu)) &= \sum_{v\in \widetilde{W}}  \Card(Y^{-1}_{I \gamma_\nu I}(u_\mu)\cap IvI)\, t^{\ell(v)}, 
\nonumber \\
\Card(Y^{-1}_{I_\pi}(u_\mu)) &= \sum_{v\in \widetilde{W}}  \Card(Y^{-1}_{I_\pi}(u_\mu)\cap IvI)\, t^{\ell(v)}.
\label{proposal1}
\end{align}
This is not quite the same as the generating series used for the type $GL_n$ case by Mellit \cite{Me17}.  Mellit's generating functions
use  $t^{\deg(v)}$, where $\deg(v)$ is the degree of the corresponding vector bundle.  Although
Hikita does not use vector bundles (and uses a Springer fiber over a semsimple element instead of a nilpotent element),
the grading parameter used by Hikita \cite[Cor. 4.6 and Theorem 4.15]{Hik12} is similar to that used by Mellit \cite{Me17}.   

Although the definition of $\Card(Y^{-1}_{I\gamma_\nu I}(u_\mu)) $
using $t^{\ell(v)}$ is irresistible, it might be that a more correct way to proceed is to use a statistic that 
connects to the representation theory of the double affine Hecke algebra and define
\begin{align}
\Card(Y^{-1}_{I\gamma_\nu I}(u_\mu)) &= \sum_{v\in \widetilde{W}}  \Card(Y^{-1}_{I \gamma_\nu I}(u_\mu)\cap IvI)\, 
t^{\mathrm{ord}(Y^{-1}_{I \gamma_\nu I}(u_\mu)\cap IvI)}, 
\nonumber \\
\Card(Y^{-1}_{I_\pi}(u_\mu)) &= \sum_{v\in \widetilde{W}}  \Card(Y^{-1}_{I_\pi}(u_\mu)\cap IvI)\, 
t^{\mathrm{ord}(Y^{-1}_{I_\pi}(u_\mu)\cap IvI)},
\label{proposal2}
\end{align}
where $\mathrm{ord}(Y^{-1}_{I \gamma_\nu I}(u_\mu)\cap IvI)$ is the maximal $k\in \ZZ_{\ge 0}$ such that
$$\hbox{if $x v I\in Y^{-1}_{I \gamma_\nu I}(u_\mu)\cap IvI$ then $x-1\in \epsilon^{-k}\mathrm{Lie}(K)$.}$$
Here $K = G(\FF_q[[\epsilon]])$, $\mathrm{Lie}(K)$ is the Lie algebra of $K$ and 
$x v I$ denotes choosing representatives of elements of $IvI$ with
$$x\quad\hbox{an element of}\quad \prod_{\beta\in \mathrm{Inv}(v)} \cX_\beta,$$ 
where $\cX_\beta$ denotes the root subgroup corresponding to a root $\beta$
and $\mathrm{Inv}(v)$ is the set of positive roots taken to negative by $v$ (the `inversion set' of $v$).
This approach is attempting to capture a statistic along the lines of what appears in 
\cite[Lemma 8.9]{GMV14}, \cite[Def.\ 8.1.1]{OY14} and
\cite[second sentence of paragraph containing (2.4.6)]{VV07}.
Though not exactly the same, the statistics $t^{\ell(v)}$ and 
$t^{\mathrm{ord}(Y^{-1}_{I \gamma_\nu I}(u_\mu)\cap IvI)}$ are closely related because $\ell(t_\lambda) = \langle \lambda, 2\rho\rangle$
for a translation $t_\lambda$ in the affine Weyl group (see \cite[(2,4,1)]{Mac03} and \cite[(5.3)]{CR22}).

With these suggestions for $\Card(Y^{-1}_{I\gamma_\nu I}(u_\mu))$ and  $\Card(Y^{-1}_{I_\pi}(u_\mu))$, the elements
$A_\mu$ and $M_\mu$ in the Hecke algebra which are defined in Section \ref{toNG} will be, up to a normalization
by the cardinality of the centralizer of $u_\mu$, analogues of integral form and modified form Macdonald polynomials
for general Lie types.
The normalization by the centralizer of $u_\mu$ is a version of the normalization by the constant $b_\lambda$ which
appears in  Macdonald's book (see \cite[Ch. III (2.7), Ch. IV (2.6) and Ch. IV (7.3')]{Mac}).

In conclusion, I'd like to apologize, and thank my readers for their patience with the wishywashyness and imprecision of this subsection.
I had begun a more in depth and careful treatment of some features but it quickly caused this paper to balloon to
an unpleasant and uncontrollable size.  
If someday I get better at computing the combinatorics of affine Lusztig varieties efficiently 
I will endeavor to tighten up these formulas.
Even better would be if someone who understands these objects better than me explains the picture properly.
It should be possible to derive Macdonald's semistandard Young tableau formula for the monomial expansion of the type 
$GL_n$ Macdonald polynomial by counting points in affine Lusztig varieties 
(as was done for the nonaffine case in \cite[Theorem 3.4]{HR99}).

\section{Examples}

\subsection{Examples for $n=2$ and $n=3$}\label{SFexs}

Integral Macdonald polynomial expansions in big Schurs:
$J_{(1)} = S_{(1)}$
$$
\begin{array}{l}
J_{(2)} = S_{(2)}+ q S_{(1^2)}, \\
J_{(1^2)} = tS_{(2)}+ S_{(1^2)},
\end{array}
\qquad\qquad
(K_{\lambda\mu}(q,t)) =
\begin{array}{c|cc} 
\lambda\backslash\mu &(2) &(1^2) \\
\hline
(2) &1 &t \\ 
(1^2) &q &1 \end{array}
$$
$$
\begin{array}{l}
J_{(3)} = S_{(3)}+ (q^2+q)S_{(21)} + q^3 S_{(1^3)}, \\
J_{(21)} = t S_{(3)}+ (1+qt)S_{(21)} + q S_{(1^3)}, \\
J_{(1^3)} = t^3 S_{(3)}+ (t^2+t)S_{(21)} + S_{(1^3)}, 
\end{array}
\qquad
(K_{\lambda\mu}(q,t)) = 
\begin{array}{c|ccc} 
\lambda \backslash \mu &(3) &(21) &(1^3) \\
\hline
(3) &1 &t  &t^3 \\ 
(21) &q+q^2 &1+qt &t+t^2 \\
(1^3) &q^3 &q &1 
\end{array}
$$
Big Schur expansions in monomial symmetric functions:  $S_{(1)} = (1-t) m_{(1)}$,
$$
\begin{array}{l}
S_{(2)} = (1-t) m_{(2)}+ (1-t)^2 m_{(1^2)}, \\
S_{(1^2)} = (-t)(1-t) m_{(2)}+ (1-t)^2 m_{(1^2)},
\end{array}
\qquad\qquad
(L_{\lambda\nu}(t)) = 
\begin{array}{c|cc} 
\lambda \backslash \nu &(2)  &(1^1) \\
\hline
(2) &1 &1 \\ 
(1^2) &-t &1 
\end{array}
$$
$$
\begin{array}{l}
S_{(3)} = (1-t) m_{(3)}+ (1-t)^2 m_{(21)} + (1-t)^3 m_{(1^3)}, \\
S_{(21)} = (-t)(1-t) m_{(3)}+ (1-t)(1-t)^2 m_{(21)} + 2(1-t)^3  m_{(1^3)}, \\
S_{(1^3)} = (-t)^2(1-t) m_{(3)}+ (-t)(1-t)^2 m_{(21)} + (1-t)^3 m_{(1^3)}, 
\end{array}
$$
$$
(L_{\lambda\nu}(t)) = 
\begin{array}{c|ccc} 
\lambda\backslash\nu &(3) &(21) &(1^3) \\
\hline
(3) &1 &1  &1 \\ 
(21) &0+(-t) &1+(-t) &1+1 \\
(1^3) &(-t)^2 &-t &1 
\end{array}
$$
Integral Macdonald polynomial expansions in monomial symmetric functions: $J_{(1)} = m_{(1)}$,
$$
\begin{array}{l}
J_{(2)} = (1-qt)(1-t)m_2+(1+q)(1-t)^2m_{(1^2)},
\\
J_{(1^2)} = (1+t)(1-t)^2m_{(1^2)},
\end{array}
\qquad
(a_{\mu\nu}(q,t)) = 
\begin{array}{c|cc} 
\mu\backslash\nu &(2) &(1^2) \\
\hline
(2) &1-qt &1+q \\ 
(1^2) &0 &1+t \end{array}
$$
\begin{align*}
J_{(3)} &= (1-qt-q^2t+q^3t^2)(1-t)m_{(3)}
+(1+q-qt+q^2-q^2t-q^3t)(1-t)^2m_{(21)}
\\
&\qquad+(1+q)(1+q+q^2)(1-t)^3m_{(1^3)},
\\
J_{(21)} &= (1-qt^2)(1-t)^2m_{(21)}+(2+q+t+2qt)(1-t)^3 m_{(1^3)},
\\
J_{(1^3)} &= (1+t)(1+t+t^2)(1-t)^3 m_{(1^3)},
\end{align*}
$$
(a_{\mu\nu}(q,t)) =
\begin{array}{c|ccc} 
\mu\backslash \nu &(3) &(21) &(1^3) \\
\hline
(3) &(1-qt)(1-q^2t) &(1-qt)(1+q+q^2)  &(1+q)(1+q+q^2) \\ 
(21) &0 &1-qt^2 &2+t+q+2qt \\
(1^3) &0 &0 &(1+t)(1+t+t^2) \end{array}
$$

Schur expansions in monomial symmetric functions:
$s_{(1)} = m_{(1)}$,
$$
\begin{array}{l}
s_{(2)} = m_{(2)}+ m_{(1^2)}, \\
s_{(1^2)} = m_{(1^2)},
\end{array}
\qquad\qquad
(K_{\lambda\mu}(0,1)) = 
\begin{array}{c|cc} 
\lambda\backslash \mu &(2)  &(1^2) \\
\hline
(2) &1 &1 \\ 
(1^2) &0 &1 
\end{array}
$$
$$
\begin{array}{l}
s_{(3)} = m_{(3)}+ m_{(21)} + m_{(1^3)}, \\
s_{(21)} = m_{(21)} + 2 m_{(1^3)}, \\
s_{(1^3)} = m_{(1^3)}, 
\end{array}
\qquad\qquad
(K_{\lambda\mu}(0,1)) = 
\begin{array}{c|ccc} 
\mu\backslash\lambda &(3) &(21) &(1^3) \\
\hline
(3) &1 &1  &1 \\ 
(21) &0 &1 &2 \\
(1^3) &0 &0 &1
\end{array}
$$
Modified Macdonald polynomial expansions in Schur functions: $\widetilde{H}_{(1)} = s_{(1)}$,
$$
\begin{array}{l}
\widetilde{H}_{(2)} = s_{(2)}+ q s_{(1^2)}, \\
\widetilde{H}_{(1^2)} = s_{(2)}+ t s_{(1^2)}, \\
\end{array}
\qquad\qquad
(t^{n(\lambda)}K_{\lambda\mu}(q,t^{-1})) =
\begin{array}{c|cc} 
\lambda\backslash\mu &(2) &(1^2) \\
\hline
(2) &1 &1 \\ 
(1^2) &q &t \end{array}
$$
$$
\begin{array}{l}
\widetilde{H}_{(3)} = s_{(3)}+ (q^2+q)s_{(21)} + q^3 s_{(1^3)}, \\
\widetilde{H}_{(21)} = s_{(3)}+ (q+t)s_{(21)} + qt s_{(1^3)}, \\
\widetilde{H}_{(1^3)} = s_{(3)}+ (t^2+t)s_{(21)} + t^3 s_{(1^3)}, 
\end{array}
\qquad\qquad
(t^{n(\lambda)}K_{\lambda\mu}(q,t^{-1})) =
\begin{array}{c|ccc} 
\lambda\backslash\mu &(3) &(21) &(1^2) \\
\hline
(3) &1 &1 &1 \\ 
(21) &q^2+q &q+t &t^2+t \\
(1^3) &q^3 &qt &t^3
\end{array}
$$
Modified Macdonald polynomial expansions in monomial symmetric functions: 
$\widetilde{H}_{(1)} = m_{(1)}$,
$$
\begin{array}{l}
\widetilde{H}_{(2)}  = m_{(2)}+(1+q)m_{(1^2)}, \\
\widetilde{H}_{(1^2)}  = m_{(2)}+(1+t)m_{(1^2)},
\end{array}
\qquad\qquad
b= (b_{\mu\nu}(q,t)) = \begin{array}{c|cc} 
\mu\backslash \nu &(2)  &(1^2) \\
\hline
(2) &1 &1+q \\ 
(1^2) &1 &1+t
\end{array}
$$
$$
\begin{array}{l}
\widetilde{H}_{(3)} = m_{(3)}+(1+q+q^2)m_{(21)}+(1+q)(1+q+q^2)m_{(1^3)}, \\
\widetilde{H}_{(21)} = m_{(3)}+(1+q+t)m_{(21)}+(1+2(q+t)+qt)m_{(1^3)}, \\
\widetilde{H}_{(1^3)} = m_{(3)}+(1+t+t^2)m_{(21)}+(1+t)(1+t+t^2)m_{(1^3)},
\end{array}
$$
$$
b = (b_{\mu\nu}(q,t)) = 
\begin{array}{c|ccc} 
\mu\backslash\nu &(3) &(21) &(1^3) \\
\hline
(3) &1 &1+q+q^2  &1+2(q^2+q)+q^3 \\ 
(21) &1 &1+q+t &1+2(q+t)+qt \\
(1^3) &1 &1+t+t^2 &1+2(t+t^2)+t^3
\end{array}
$$

\subsection{Hecke algebra elements and expansion-contraction}

The irreducible characters of $H$ and the unipotent irreducible characters of $G$ are given by
$$\chi^\lambda_H(T_{\gamma_\nu}) = L_{\lambda\nu}(q^{-1})q^{n-\ell(\nu)}
\qquad\hbox{and}\qquad
\chi^\lambda_G(u_\mu) = q^{n(\mu)}K_{\lambda\mu}(0,q^{-1})
$$
and
$$
A_{\mu\nu} = \Card(Y^{-1}_{B\gamma_\nu B}(u^{-1}_\mu)) 
=(q^{n(\mu)}a_{\mu\nu}(0,q^{-1})q^{n-\ell(\nu)}).
$$
Since
$$
A_{\mu\nu}= (\tr(u_\mu T_{\gamma_\nu}, \mathbf{1}_B^G))
= \sum_\lambda \chi^\lambda_G(u_\mu)\chi^\lambda_H(T_{\gamma_\nu}) 
\qquad\hbox{then}\qquad
A = \chi_G^t \chi_H.
$$

The Geck-Rouquier basis elements $\kappa_\nu$, 
the minimal central idempotents $z_\lambda^H$, and the central elements $A_\mu$ are
$$\kappa_\nu = \sum_w \kappa_{\nu,w} q^{-\ell(w)}T_{w^{-1}},
\qquad
z^H_\lambda = \frac{\chi^\lambda_G(1)}{\vert G/B\vert}\sum_\nu \chi^\lambda_H(T_{\gamma_\nu})\kappa_\nu
\qquad\hbox{and}\qquad
A_\mu = \sum_\nu A_{\mu\nu}\kappa_\nu.
$$

\subsubsection{Type $GL_2(\FF_q)$}

The parabolic projectors and the contraction matrix are
$$
\begin{array}{l}
\mathbf{1}_{(1^2)} = 1, \\
\mathbf{1}_{(2)} = 1+T_{s_1}, 
\end{array}
\qquad\hbox{and}\qquad
(C_{w,\pi})_{w\in W,\pi\in \cP} = \begin{array}{c|cccc}
w\backslash \pi &(2) &(1^2) \\
\hline
s_1 &1 &0 \\
1 &1  &1 
\end{array}
$$
The Geck-Rouquier elements and the expansion matrix are
$$
\begin{array}{l}
\kappa_{s_1} = 1\cdot q^{-1}T_{s_1} +0\cdot T_1 \\
\kappa_1 = 0\cdot q^{-1}T_{s_1} + T_1, 
\end{array}
\qquad
\kappa = (\kappa_{\nu,w}) =
\begin{array}{c|cc}
\nu\backslash w &s_1 &1 \\
\hline
s_1 &1 &0 \\
1 &0 &1 
\end{array}
$$
The minimal central idempotents in $H$ and the character table of the Hecke algebra are
$$
\begin{array}{l}
z^H_{(2)}  = \frac{1}{1+q}(1+T_{s_1}) = \frac{1}{1+q}(q\kappa_{s_1}+\kappa_1),
\\
z^H_{(1^2)} = \frac{1}{1+q}(q-T_{s_1}) = \frac{q}{1+q}(-\kappa_{s_1}+\kappa_1) 
\end{array}
\qquad\quad
\chi_H = (\chi^\lambda_H(T_{\gamma_\nu}))
= \begin{array}{c|cc}
\lambda \backslash \nu &(2) &(1^2) \\
\hline
(2) &q &1 \\
(1^2) &-1 &1 
\end{array}
$$
The unipotent character table of $G$ is
$$
\chi_G = (\chi^\lambda_G(u_\mu))
= \begin{array}{c|cc}
\lambda \backslash \mu &(2) &(1^2) \\
\hline
(2) &1 &1 \\
(1^2) &0 &q 
\end{array}
$$
The central elements $A_\mu$ are
$$
\begin{array}{l}
A_{(2)} = q\kappa_{s_1}+\kappa_{1} = T_{s_1}+1,
\\
A_{(1^2)} = (1+q) \kappa_1 = [2]T_1,
\end{array}
\qquad
(A_{\mu\nu}(q)) 
= (\tr(u_\mu T_{\gamma_\nu}, \mathbf{1}_B^G))
= \begin{array}{c|cc}
\mu\backslash \nu &(2) &(1^2) \\
\hline
(2) &q &1 \\
(1^2) &0 &q+1 
\end{array}
\qquad\qquad
$$

\subsubsection{Type $GL_3(\FF_q)$}

The parabolic projectors and the contraction matrix are
$$
\begin{array}{l}
\mathbf{1}_{(3)} = T_{s_1s_2}+T_{s_2s_1}+T_{s_1s_2s_1}+T_{s_1}+T_{s_2}+T_1, \\
\mathbf{1}_{(21)} = T_{s_1}+T_1, \\
\mathbf{1}_{(1^3)} = T_1,
\end{array}
\quad\hbox{and}\quad
(C_{w,\pi}) = \begin{array}{c|ccc}
w\backslash \pi &(3) &(21) &(1^3) \\
\hline
s_1s_2 &1 &0  &0 \\
s_2s_1 &1 &0 &0 \\
s_1s_2s_1 &1 &0  &0 \\
s_2 &1 &0  &0 \\
s_1 &1 &1  &0 \\
1 &1 &1  &1 
\end{array}
$$
The Geck-Rouquier elements and the expansion matrix are
$$
\begin{array}{l}
\kappa_{s_1s_2} = 1\cdot q^{-2}T_{s_1s_2}+1\cdot q^{-2}T_{s_2s_1}+ (q-1) q^{-3}T_{s_1s_2s_1}, \\
\kappa_{s_1} = q\cdot q^{-3}T_{s_1s_2s_1} + 1\cdot q^{-1}T_{s_1} + 1\cdot q^{-1}T_{s_2},  \\
\kappa_1 = T_1,
\end{array}
\qquad
\begin{array}{c|cccccc}
\nu\backslash w &s_1s_2 &s_2s_1 &s_1s_2s_1 &s_2 &s_1 &1 \\
\hline
s_1s_2 &1 &1 &q-1 &0 &0 &0 \\
s_1 &0 &0 &q &1 &1 &0 \\
1 &0 &0 &0 &0 &0 &1 
\end{array}
$$
The minimal central idempotents and the character table of the Hecke algebra are
\begin{align*}
\begin{array}{l}
z^H_{(1^3)} = \frac{1}{[3]!}(q^2\kappa_{s_1s_2}+q\kappa_{s_1}+\kappa_1), \\
z^H_{(21)} 
= \frac{q[2]}{[3]!}( - q\kappa_{s_1s_2} +(q-1)\kappa_{s_1} +2\kappa_1 ), \\
z^H_{(3)} = \frac{q^3}{[3]!}(\kappa_{s_1s_2}-\kappa_{s_1} + \kappa_1), 
\end{array}
\qquad
\chi_H=
(\chi^\lambda_H(T_{\gamma_\nu})) = 
\begin{array}{c|ccc}
\lambda \backslash \nu &(3) &(21) &(1^3) \\
\hline
(3) &q^2 &q &1 \\
(21) &-q &q-1 &2 \\
(1^3) &1 &-1 &1
\end{array}
\end{align*}
The unipotent character table of $G$ is
$$
\chi_G =
(\chi^\lambda_G(u_\mu)) = 
\begin{array}{c|ccc}
\lambda \backslash \mu &(3) &(21) &(1^3) \\
\hline
(3) &1 &1 &1 \\
(21) &0 &q &q^2+q \\
(1^3) &0 &0 &q^3
\end{array}
$$
The central elements $A_\mu\in Z(H)$ are
\begin{align*}
A_{(3)} &= T_{s_1s_2}+T_{s_2s_1}+T_{s_1s_2s_1}+T_{s_1}+T_{s_2}+T_1 
= q^2\kappa_{s_1s_2}+q\kappa_{s_1}+\kappa_1, \\
A_{(21)}  &= T_{s_1s_2s_1}  + qT_{s_1}+qT_{s_2} + (1+2q)T_1 
= q^2 \kappa_{s_1} +(1+2q)\kappa_1, \\
A_{(1^3)} &= [3]! T_1 = [3]! \kappa_1,
\end{align*}
and the table of bitraces is
$$
A = (A_{\mu\nu}(q)) = (\tr(u_\mu T_{\gamma_\nu}, \mathbf{1}_B^G) = (q^{n(\mu)+n-\ell(\nu)}a_{\mu\nu}(0,q^{-1})) = 
\begin{array}{c|ccc}
\mu\backslash\nu &(3) &(21) &(1^3) \\
\hline
(3) &q^2 &q &1 \\
(21) &0 &q^2 &2q+1 \\
(1^3) &0 &0 &[3][2]
\end{array}
$$
Let us verify that
$$A = (A_{\mu\nu}) = \chi_G^t \chi_H 
= \begin{pmatrix}  1 &0 &0  \\ 1 &q &0 \\ 1 &q^2+q &q^3 \end{pmatrix}
\begin{pmatrix} q^2 &q &1 \\ -q &q-1 &2 \\ 1 &-1 &1 \end{pmatrix}
=\begin{pmatrix} q^2 &q &1 \\ 0 &q^2 &2q+1 \\ 0 &0 &[3][2] \end{pmatrix}
\quad\hbox{and}
$$
$$\chi_H = L_{\lambda\nu}(q^{-1})\mathrm{diag}(q^{n-\ell(\nu)})
=\begin{pmatrix} 1 &1 &1 \\ -q^{-1} &1-q^{-1} &2 \\ q^{-2} &-q &1 \end{pmatrix}
\begin{pmatrix} q^2 &0 &0 \\ 0 &q &0 \\ 0 &0 &1 \end{pmatrix}
=\begin{pmatrix} q^2 &q &1 \\ -q &q-1 &2 \\ 1 &-1 &1 \end{pmatrix}.
$$
Multiplying $A$ by the matrix $\kappa = (\kappa_{\nu,w})$ gives
$$
(A_{\mu w}(q)) = (\#Y^{-1}_{BwB}(u_\mu)) = (\tr(u_\mu T_w, \mathbf{1}_B^G) =
\begin{array}{c|cccccc}
\mu \backslash w &s_1s_2 &s_2s_1 &s_1s_2s_1 &s_2 &s_1 &1 \\
\hline
(3) &q^2 &q^2 &q^3 &q &q &1 \\
(21) &0 &0 &q^3 &q^2 &q^2 &2q+1 \\
(1^3) &0 &0 &0 &0 &0 &[3][2] 
\end{array}
$$
since $(q-1)\cdot q^2+q\cdot q = q^3$, $(q-1)\cdot 0 + q\cdot q^2 = q^3$ and $(q-1)\cdot 0 + q\cdot 0 = 0$.
Note that the row sums of this matrix are all $[3][2]$.  Then, multiplying by the contraction matrix gives
$$
(\#Y^{-1}_{P_\pi}(u_\mu))
=
\begin{array}{c|cccc}
\mu \backslash \pi &Y^{-1}_{(1^3)} &Y^{-1}_{(21)} &Y^{-1}_{(12)} &Y^{-1}_{(3)} \\
\hline
(3) &1 &1+q &1+q &[3][2] \\
(21) &2q+1 &q^2+2q+1 &q^2+2q+1 &[3][2] \\
(1^3) &[3][2] &[3][2] &[3][2] &[3][2] 
\end{array}
$$

\subsubsection{Check that $M\kappa CDFL=1$ for $n=2$ and $n=3$}
The expansion and contraction matrices for $n=2$ are 
$$\kappa = (\kappa_{\nu,w})_{\nu\in \cW, w\in W} = \begin{array}{c|cc}
\nu \backslash w &s_1 &1 \\
\hline
s_1 &1 &0 \\
1 &0  &1 
\end{array}
\qquad\hbox{and}\qquad
(C_{w,\pi})_{w\in W,\pi\in \cP} = \begin{array}{c|cccc}
w\backslash \pi &(2) &(1^2) \\
\hline
s_1 &1 &0 \\
1 &1  &1 
\end{array}
$$
Using notations as in the proof of Theorem \ref{atob},
the product
$$M\kappa CD = \begin{pmatrix} q &0 \\ 0 &1 \end{pmatrix}
\begin{pmatrix} 1 & 0 \\ 0 &1 \end{pmatrix} \begin{pmatrix} 1 &0 \\ 1 &1 \end{pmatrix}
\begin{pmatrix} \frac{1}{[2]} &0 \\ 0 &1 \end{pmatrix}
= \begin{pmatrix} \frac{q}{[2]} &0 \\ \frac{1}{[2]} &1 \end{pmatrix}
$$
is the inverse of
$$
FL = \begin{pmatrix} 1 &-1 \\ 0 &1\end{pmatrix}\begin{pmatrix} 1 &1 \\ -q^{-1} &1 \end{pmatrix}
= \begin{pmatrix} 1+q^{-1} &0 \\ -q^{-1} &1 \end{pmatrix}.$$

From Section \ref{SFexs},
$$
(a_{\mu\nu}(q,t)) = 
\begin{array}{c|cc} 
\mu\backslash\nu &(2) &(1^2) \\
\hline
(2) &1-qt &1+q \\ 
(1^2) &0 &1+t \end{array}
\qquad\hbox{so that}\qquad
(a_{\mu\nu}(t,q^{-1})) = 
\begin{array}{c|cc} 
\mu\backslash\nu &(2) &(1^2) \\
\hline
(2) &1-tq^{-1} &1+t \\ 
(1^2) &0 &1+q^{-1} \end{array}
$$
and
$$
(q^{n(\mu)}a_{\mu\nu}(t,q^{-1})) = 
\begin{array}{c|cc} 
\mu\backslash\nu &(2) &(1^2) \\
\hline
(2) &1-tq^{-1} &1+t \\ 
(1^2) &0 &1+q \end{array}
$$
and
$$
(q^{n(\mu)}a_{\mu\nu}(t,q^{-1})q^{n-\ell(\nu)}) = 
\begin{array}{c|cc} 
\mu\backslash\nu &(2) &(1^2) \\
\hline
(2) &q(1-tq^{-1}) &1+t \\ 
(1^2) &0 &1+q \end{array}
$$
From Section \ref{SFexs},
$$
b= (b_{\mu\nu}(q,t)) = \begin{array}{c|cc} 
\mu \backslash \nu &(2)  &(1^2) \\
\hline
(2) &1 &1+q \\ 
(1^2) &1 &1+t
\end{array}
\qquad\hbox{so that}\qquad
(b_{\mu\nu}(t,q)) = \begin{array}{c|cc} 
\mu \backslash \nu &(2)  &(1^2) \\
\hline
(2) &1 &1+t \\ 
(1^2) &1 &1+q
\end{array}
$$

The expansion matrix and contraction matrices for $n=3$ are
$$(\kappa_{\nu,w}) = \begin{array}{c|cccccc}
\nu\backslash w &s_1s_2 &s_2s_1 &s_1s_2s_1 &s_2 &s_1 &1 \\
\hline
s_1s_2 &1 &1 &q-1 &0 &0 &0 \\
s_1 &0 &0 &q &1 &1 &0 \\
1 &0 &0 &0 &0 &0 &1 
\end{array}
\quad\hbox{and}\quad
(C_{w,\pi}) = \begin{array}{c|ccc}
w\backslash \pi &(3) &(21) &(1^3) \\
\hline
s_1s_2 &1 &0  &0 \\
s_2s_1 &1 &0 &0 \\
s_1s_2s_1 &1 &0  &0 \\
s_2 &1 &0  &0 \\
s_1 &1 &1  &0 \\
1 &1 &1  &1 
\end{array}
$$
Using notations as in the proof of Theorem \ref{atob},
the product
\begin{align*}
M\kappa CD &= \begin{pmatrix}q^2 &0 &0 \\ 0 &q &0 \\ 0 &0 &1 \end{pmatrix}
\begin{pmatrix} 1 &1 &q-1 &0 &0 &0 \\ 0 &0 &q &1 &1 &0 \\ 0 &0 &0 &0 &0 &1 \end{pmatrix}
\begin{pmatrix} 1 &0 &0 \\ 1 &0 &0 \\ 1 &0 &0 \\ 1 &0 &0 \\ 1 &1 &0 \\ 1 &1 &1 \end{pmatrix}
\begin{pmatrix} \frac{1}{[3][2]} &0 &0 \\ 0 &\frac{1}{[2]} &0 \\ 0 &0 &1 \end{pmatrix} \\
&= 
\begin{pmatrix}q^2 &0 &0 \\ 0 &q &0 \\ 0 &0 &1 \end{pmatrix}
\begin{pmatrix} q+1 &0 &0 \\ q+2 &1 &0 \\ 1 &1 &1 \end{pmatrix}
\begin{pmatrix} \frac{1}{[3][2]} &0 &0 \\ 0 &\frac{1}{[2]} &0 \\ 0 &0 &1 \end{pmatrix} 
=
\begin{pmatrix} \frac{q^2}{[3]} &0 &0 \\ \frac{q^2+2q}{[3][2]} &\frac{q}{[2]} &0 \\ \frac{1 }{ [3][2]} &\frac{1}{ [2] } &1 \end{pmatrix}
\end{align*}
is the inverse of
$$
FL  = 
\begin{pmatrix} 1 &-1 &1 \\ 0 &1 &-2 \\ 0 &0 &1 \end{pmatrix}
\begin{pmatrix} 1 &1 &1 \\ -q^{-1} &1-q^{-1} &2 \\ q^{-2} &-q^{-1} &1 \end{pmatrix}
=
\begin{pmatrix} q^{-2}[3] &0 &0 \\ -q^{-2}(q+2) &q^{-1}(q+1) &0 \\ q^{-2} &-q^{-1} &1 \end{pmatrix}.
$$


\end{document}